\documentclass[parskip=never,abstract=true]{scrartcl}

\usepackage{csquotes}

\usepackage[T1]{fontenc}
\usepackage[utf8]{inputenc}
\usepackage{lmodern}
\usepackage{yfonts}
\usepackage{textcomp}
\usepackage{setspace}
\usepackage[a4paper,left=2.5cm,right=2cm,top=2.5cm,bottom=2.5cm]{geometry}

\usepackage{extarrows}
\usepackage{etoolbox}
\usepackage{xparse}

\usepackage{mathtools}
\usepackage{amsthm,thmtools}
\usepackage{dsfont}
\let\mathbb\mathds
\usepackage{amssymb}
\usepackage[scr]{rsfso} 
\usepackage{bm}
\usepackage{euscript}
\usepackage{amsbsy}
\DeclareMathAlphabet\mathbfcal{OMS}{cmsy}{b}{n}

\usepackage{microtype}
\usepackage{authblk}
\usepackage{graphicx} 
\usepackage{tikz}
\usetikzlibrary{cd,quotes}
\usetikzlibrary{decorations.markings}
\usepackage{pgfplots}
\pgfplotsset{compat=1.13}
\usepackage{booktabs}
\usepackage{float}
\usepackage{environ}
\usepackage{xcolor,soul}
\usepackage[shortlabels]{enumitem}
\usepackage{eso-pic}

\usepackage[normalem]{ulem}

\usepackage{hyperref} 
\hypersetup{%
	colorlinks,%
	linkcolor={red!60!black},%
	citecolor={blue!50!black},%
	urlcolor={blue!80!black}%
}

\usepackage{tikz-cd}
\tikzset{Rightarrow/.style={double equal sign distance,>={Implies},->},
	triple/.style={-,preaction={draw,Rightarrow}},
	quadruple/.style={preaction={draw,Rightarrow,shorten >=0pt},shorten >=1pt,-,double,double
		distance=0.2pt}}

\def\on{\operatorname}

\def\CC{\mathbb{C}}
\def\B{\EuScript{B}}
\def\C{\EuScript{C}}
\def\D{\EuScript{D}}
\def\E{\EuScript{E}}
\def\X{\EuScript{X}}
\def\Y{\EuScript{Y}}

\def\DD{\mathbb{D}}

\def\Fun{\on{Fun}}
\def\Cat{\on{Cat}}
\def\iCat{\EuScript{C}\!\on{at}}

\def\Hom{\on{Hom}}
\def\D{\EuScript{D}}
\def\Nerv{\on{N}}

\SetEnumitemKey{axioms}{itemsep=0pt,align=left,itemindent=!,
                        listparindent=\parindent,parsep=\parskip,
                        before=,
                        label={\Roman*}}

\SetEnumitemKey{claims}{itemsep=0pt,align=left,itemindent=!,
                        listparindent=\parindent,parsep=\parskip,
                        before=,
                        label={\Roman*}}

\DeclarePairedDelimiterX\set[1]{\lbrace}{\rbrace}{#1}
\newlist{implications}{description}{1} 
\setlist[implications]{itemsep=0pt,leftmargin=\parindent}
\NewDocumentCommand\implication{o}
  {\IfValueTF{#1}
    {\auximplication#1\relax}
    {\item[\normalfont($\,\Rightarrow\,$)]}}
\NewDocumentCommand\auximplication{u-u\relax}
  {\item[\normalfont(#1)$\,\Rightarrow\,$(#2)]}

\makeatletter
\newcounter{diagram}[section]
\def\thediagram{\thesection.\arabic{diagram}}
\def\ftype@diagram{4}
\def\ext@diagram{diag}
\def\fnum@diagram{Diagram~\thediagram}
\def\fs@diagram{htbp!}
\ExplSyntaxOn
\NewDocumentEnvironment{diagram}{O{htbp!}m}
  {\@float{diagram}[#1]\centering}
  {
   
   \caption{}
   \label{#2}
   \end@float
  }

\newcounter{subdiagram}[diagram]
\def\thesubdiagram{\thediagram.\arabic{subdiagram}}
\NewEnviron{multidiagram}{\expandafter\domultidiagram\BODY\enddomultidiagram}
\NewDocumentCommand\domultidiagram{omu\enddomultidiagram}
 {
  \IfValueTF{#1}{\diagram[#1]}{\diagram}{}
   \refstepcounter{diagram}
   \centering
    \seq_clear:N \l_tmpb_seq
    \seq_set_split:Nnn \l_tmpa_seq { \next } { #3 }
    \seq_map_inline:Nn \l_tmpa_seq
     {
      \seq_put_right:Nn \l_tmpb_seq
       {
        \begin{tabular}[b]{@{}c@{}}
         ##1 \\[3ex]
         \refstepcounter{subdiagram}
         \label{#2\othercolon\the\value{subdiagram}}
         Diagram~\thesubdiagram 
        \end{tabular}
       }
     }
    \seq_use:Nn \l_tmpb_seq { \qquad }
   \let\label\@gobble
   \let\caption\@gobble
  \enddiagram
 }
\ExplSyntaxOff
\def\othercolon{:}
\makeatother

\declaretheoremstyle[bodyfont=\itshape,notefont=\bfseries]{abellanA}
\declaretheoremstyle[notefont=\bfseries]{abellanB}

\declaretheorem[style=abellanA,numberwithin=section,name={Theorem}]{theorem}

\declaretheorem[style=abellanA,numberlike=theorem,name={Lemma}]{lemma}
\declaretheorem[style=abellanB,numberlike=theorem,name={Definition}]{definition}
\declaretheorem[style=abellanB,numberlike=theorem,name={Remark}]{remark}

\declaretheorem[style=abellanA,numberlike=theorem,name={Proposition}]{proposition}

\declaretheorem[style=abellanB,numbered=no,name={Notation}]{notation}
\declaretheorem[style=abellanA,numberlike=theorem,name={Corollary}]{corollary}

\docsvlist{theorem,lemma,definition,remark,proposition,example,corollary}

\let\epsilon\varepsilon
\let\isom\simeq

\newcommand*\tensor{\otimes}

\newcommand*\mathblank{\mathord{-}}

\DeclareMathOperator*\colimdag{colim^\dagger}
\DeclareMathOperator*\colimsharp{colim^\sharp}
\DeclareMathOperator*\colimflat{colim^\flat}
\DeclareMathOperator*\colimnat{colim^\natural}
\DeclareMathOperator*\colim{colim}

\def\msSet{{\on{Set}_{\Delta}^+}}

\def\Cat{\on{Cat}}

\usepackage[bbgreekl]{mathbbol}

\makeatletter
\newcommand{\fixed@sra}{$\vrule height 2\fontdimen22\textfont2 width 0pt\rightarrow$}
\newcommand{\shortarrowup}[1]{%
	\mathrel{\text{\rotatebox[origin=c]{65}{\fixed@sra}}}
}
\newcommand{\shortarrowdown}[1]{%
	\mathrel{\text{\rotatebox[origin=c]{250}{\fixed@sra}}}
}
\makeatother

\newcommand{\downslash}{\!\shortarrowdown{1}}

\def\lra{\longrightarrow}
\def\lla{\longleftarrow}

\def\llra{\def\arraystretch{.1}\begin{array}{c} \lra \\ \lla \end{array}}

\newcommand{\goth}{\textfrak}

\tikzcdset{
  arrows={},
  nat/.style={Rightarrow},
  nat1/.style={nat,shift left=.8ex},
  nat2/.style={nat,shift right=.8ex,swap},
  map/.style={},
  map1/.style={map,shift left=.7ex},
  map2/.style={map,shift right=.7ex,swap},
  mapA/.style={map,shift left=.5ex},
  mapB/.style={map,shift left=.5ex},
  unique/.style={line cap=round,densely dotted},
  incl/.style={hookrightarrow},
  mono/.style={rightarrowtail},
  epi/.style={twoheadrightarrow},
  iso/.style={map,sloped,"{\tikz[x=.6ex,y=.3ex]{\draw[-](0,0)sin(1,1)cos(2,0)sin(3,-1)cos(4,0);}}"},
  line/.style={dash},
  line1/.style={line,shift left=.3ex},
  line2/.style={line,shift right=.3ex,swap},
}

\def\op{{\on{op}}}

\makeatletter
\newcommand*\dirlim{\mathop{\mathpalette\varlim@{\rightarrowfill@\scriptscriptstyle}}\nmlimits@}
\newcommand*\prolim{\mathop{\mathpalette\varlim@{\leftarrowfill@\scriptscriptstyle}}\nmlimits@}
\makeatother

\def\llra{\def\arraystretch{.1}\begin{array}{c} \lra \\ \lla \end{array}}

\newcommand{\nat}{\Rightarrow}

\tikzset{
  abellanarrows/.style={line cap=round,line join=round,line width=.4pt},
  abellanarrowlength/.store in=\abellanarrowlength,
}

\ExplSyntaxOn

\cs_generate_variant:Nn \tl_replace_all:Nnn { Nf }

\NewDocumentCommand \func { s O{} m }
 {
  \group_begin:
   \IfBooleanTF{#1}
    { \keys_set:nn { abellan / func } { aligned = true , #2 } }
    { \keys_set:nn { abellan / func } {#2} }
   \abellan_func:n {#3}
  \group_end:
 }
\NewDocumentCommand \arr { s o m }
 {
  \IfBooleanF{#1}
   { \bool_if:NT \l_abellan_aligned_bool { & } }
  \abellan_arr:n {#3}
 }
\NewDocumentCommand \addarr { o m m }
 {
  \keys_set:nn { abellan / func / addarrow } { name = {#2} , #3 }
  \tl_clear:N \l_abellan_arrname_tl
 }
\NewDocumentCommand \setupfunc { m } { \keys_set:nn { abellan / func } {#1} }

\bool_new:N \l_abellan_aligned_bool
\tl_new:N \l_abellan_func_tl
\tl_new:N \l_abellan_arrname_tl
\tl_new:N \l_abellan_arrmode_tl
\tl_new:N \g_abellan_func_substitutions_tl
\quark_new:N \q_abellan

\keys_define:nn { abellan / func }
 {
  aligned .bool_set:N = \l_abellan_aligned_bool ,
  aligned .initial:n = false ,
  mode .choices:nn = { base , tikz } { \tl_set_eq:NN \l_abellan_arrmode_tl \l_keys_choice_tl } ,
  mode .initial:n = base ,
  tikzarrowlength .code:n = \tikzset{abellanarrowlength={#1}} ,
  tikzarrowlength .initial:n = 2em ,
  tikzarrowstyle .code:n = \tikzset{abellanarrows/.append ~ style={#1}} ,
 }
\keys_define:nn { abellan / func / addarrow }
 {
  name .tl_set:N = \l_abellan_arrname_tl ,
  base .code:n =
    \abellan_addarrow:nnww { \l_abellan_arrname_tl } { base } #1 \q_abellan ,
  tikz .code:n =
    \abellan_addarrow:nnww { \l_abellan_arrname_tl } { tikz } #1 \q_abellan ,
  substitutions .code:n =
   \tl_map_inline:nn {#1}
    {
     \tl_gput_right:Nx \g_abellan_func_substitutions_tl
      {
       \exp_not:n { \tl_replace_all:Nnn \l_abellan_func_tl {##1} }
        { \arr{\l_abellan_arrname_tl} }
      }
    } ,
 }
\cs_new_protected:Npn \abellan_func:n #1
 {
  \resetdynamicto
  \tl_set:Nn \l_abellan_func_tl {#1}
  \tl_replace_all:Nfn \l_abellan_func_tl
   { \char_generate:nn { `: } { 12 } } { \colon }
  \bool_if:NTF \l_abellan_aligned_bool
   { \tl_replace_all:Nnn \l_abellan_func_tl { ; } { \\ } }
   { \tl_replace_all:Nnn \l_abellan_func_tl { ; } { ,\ } }
  \tl_use:N \g_abellan_func_substitutions_tl
  \bool_if:NT \l_abellan_aligned_bool { \! \begin {aligned} \scan_stop: }
  \tl_use:N \l_abellan_func_tl
  \bool_if:NT \l_abellan_aligned_bool { \\ \end {aligned} }
 }
\NewDocumentCommand \abellan_addarrow:nnww { m m O{} u\q_abellan }
 {
  \exp_args:Nc \NewDocumentCommand { abellan_arr_#1_#2:w } { #3 }
   {
    \use:c { abellan_arr_ \l_abellan_arrmode_tl :n } { #4 }
   }
 }
\cs_new_protected:Npn \abellan_arr:n #1
 {
  \use:c { abellan_arr_#1_\l_abellan_arrmode_tl :w }
 }
\cs_new_protected:Npn \abellan_arr_base:n #1
 {
  \mathrel{#1}
 }
\cs_new_protected:Npn \abellan_arr_tikz:n #1
 {
  \mathrel{\vcenter{\hbox{\tikz[abellanarrows]{#1}}}}
 }
\ExplSyntaxOff

\addarr{monomorphism}
 {
  substitutions = {>->}{ mono }{\rightarrowtail}{\mono} ,
  base = [o]{\IfValueTF{#1}{\rightarrowtail{#1}}{\mono}} ,
  tikz = [o]{\draw[>->,inner sep=0pt]
     (0,0)--({max(\abellanarrowlength,.5em\IfValueT{#1}{+width("$\scriptstyle#1$")})},0)
     \IfValueT{#1}{node[midway,above=.7ex]{\smash{$\scriptstyle#1$}}
     node[midway,below=.7ex]{}}
     ;} ,
  }
\addarr{mapsto}
 {
  substitutions = {\mapsto}{ mapsto } ,
  base = \mapsto ,
  tikz = {\draw[|->](0,0)--(\abellanarrowlength,0);} ,
 }
\addarr{rrarrows} 
 {
  substitutions = {\rightrightarrows}{ doublerr }{\doublerr} ,
  base = \rightrightarrows ,
  tikz = {\draw[>={To[width=.8ex,length=.4ex]},->](0,0)--(\abellanarrowlength,0);
          \draw[yshift=.8ex,>={To[width=.8ex,length=.4ex]},->](0,0)--(\abellanarrowlength,0);} ,
 }
\addarr{inclusion}
 {
  substitutions = {\hookrightarrow}{\incl}{ incl } ,
  base = [o]{\IfValueTF{#1}{\hookrightarrow{#1}}{\incl}} ,
  tikz = {\draw[{Hooks[right]}->](0,0)--(\abellanarrowlength,0);} ,
 }
\addarr{natural}
 {
 substitutions = {\Rightarrow}{=>}{ nat }{\nat} ,
  base = [o]{\IfValueTF{#1}{\xRightarrow{#1}}{\nat}} ,
  tikz = [o]{\draw[line cap=butt,double equal sign distance,-Implies,inner sep=0pt]
    (0,0)--({max(\abellanarrowlength,.3em\IfValueT{#1}{+width("$\scriptstyle#1$")})},0)
    \IfValueT{#1}{node[midway,above=.7ex]{\smash{$\scriptstyle#1$}}
    node[midway,below=.7ex]{}}
    ;} ,
 }
\addarr{epimorphism} 
 {
  substitutions = {\twoheadrightarrow}{->>}{ epi }{\epi},
   base = [o]{\IfValueTF{#1}{\xtwoheadrightarrow{#1}}{\epi}} ,
   tikz = [o]{\draw[->>,inner sep=0pt]
      (0,0)--({max(\abellanarrowlength,.5em\IfValueT{#1}{+width("$\scriptstyle#1$")})},0)
      \IfValueT{#1}{node[midway,above=.7ex]{\smash{$\scriptstyle#1$}}
      node[midway,below=.7ex]{}}
      ;} ,
   }
\addarr{to}
 {
  substitutions = {\to}{\rightarrow}{ to }{ funct }{ map }{->} ,
  base = [o]{\IfValueTF{#1}{\xrightarrow{#1}}{\to}} ,
  tikz = [o]{\draw[->,inner sep=0pt]
    (0,0)--({max(\abellanarrowlength,.5em\IfValueT{#1}{+width("$\scriptstyle#1$")})},0)
    \IfValueT{#1}{node[midway,above=.7ex]{\smash{$\scriptstyle#1$}}
    node[midway,below=.7ex]{}}
    ;} ,
 }
\newcommand*\resetdynamicto
  {\gdef\dynamicto{\arr*{to}\gdef\dynamicto{\arr*{mapsto}}}}
\resetdynamicto

\setupfunc{mode=tikz}

\ExplSyntaxOn
\NewDocumentCommand \printheader { m o m }
 {
  \par\noindent
  \begin{minipage}[t]{\textwidth}\noindent
  
  \begin{tabular}[t]{ll}
     & \keyval_parse:NNn \abellan_printname:n \abellan_printnamemail:nn { #3 } 
  \end{tabular}
  \vspace{.4cm}
  \end{minipage}
  \begin{center}\Large\bfseries
   #1 \IfValueT{#2}{\\[1ex] \large #2}
  \end{center}
  \vspace{.6cm}
 }
\tl_new:N \g_abellan_autores_tl
\tl_new:N \g_abellan_asignatura_tl
\cs_new_protected:Npn \abellan_printnamemail:nn #1 #2
 {
  #1 \\ & \quad\texttt{#2} \\ &
 }
\cs_new_protected:Npn \abellan_printname:n #1
 {
  #1 \\ &
 }
\ExplSyntaxOff

\makeatletter
\tikzcdset{
	open/.code     = {\tikzcdset{hook, circled};},
	open'/.code    = {\tikzcdset{hook', circled};},
	circled/.code  = {\tikzcdset{markwith = {\draw (0,0) circle (.375ex);}};},
	markwith/.code ={
		\pgfutil@ifundefined%
		{tikz@library@decorations.markings@loaded}%
		{\pgfutil@packageerror{tikz-cd}{You need to say %
				\string\usetikzlibrary{decorations.markings} to use arrows with markings}{}}{}%
		\pgfkeysalso{/tikz/postaction = {
				/tikz/decorate,
				/tikz/decoration={markings, mark = at position 0.5 with {#1}}}
		}
	},
}
\makeatother

\def\scr{\EuScript}

\title{Marked colimits and higher cofinality}
\author{Fernando Abell\'an Garc\'ia}
\date{}

\begin{document}
    \maketitle
	
  		\begin{abstract}
  			Given a marked $\infty$-category $\EuScript{D}^{\dagger}$ (i.e. an $\infty$-category equipped with a specified collection of morphisms) and a functor $\func{F: \D \to \mathbb{B}}$ with values in an $\infty$-bicategory, we define $\colimdag F$, the marked colimit of $F$. We provide a definition of weighted colimits in $\infty$-bicategories when the indexing diagram is an $\infty$-category and show that they can be computed in terms of marked colimits. In the maximally marked case $\D^{\sharp}$, our construction retrieves the $\infty$-categorical colimit of $F$ in the underlying $\infty$-category $\scr{B}\subseteq \mathbb{B}$. In the specific case when $\mathbb{B}=\goth{Cat}_{\infty}$, the $\infty$-bicategory of $\infty$-categories and $\D^{\flat}$ is minimally marked, we recover the definition of lax colimit of Gepner-Haugseng-Nikolaus. We show that a suitable $\infty$-localization of the associated coCartesian fibration $\on{Un}_{\D}(F)$ computes $\colimdag F$. Our main theorem is a characterization of those functors of marked $\infty$-categories $\func{f:\C^{\dagger} \to \D^{\dagger}}$ which are marked cofinal. More precisely, we provide sufficient and necessary criteria for the restriction of diagrams along $f$ to preserve marked colimits. 
  		\end{abstract}
		
	\tableofcontents	
	
\newpage		
\section{Introduction}
The theory of $\infty$-categories is, by now, well-established as an excellent way to treat coherence and higher homotopical data. However, the mere presence of this higher data means that many properties which can be explored most easily by explicit computation in the ordinary categorical setting are best accessed by universal properties in the $\infty$-categorical setting. Consequently, general constructions exhibiting universal properties take on an even greater importance in the study of higher category theory. Among such constructions, the theories of limits and colimits form an essential core around which many of the results in higher category theory are built. In the developing theory of $\infty$-bicategories, as in its strict 1-categorical analogue, these theories must be extended to allow for \emph{laxness} --- loosely, allowing cones over a functor to commute up to non-invertible 2-morphism.

This paper is a new entry in this story, dealing with lax colimits which depend on a collection of marked morphisms in the source. These morphisms can be viewed as “controlling the laxness“ of the colimit in question. Such collections of marked morphisms arise throughout the study of higher categories --- in localizations, cartesian fibrations, etc. The theory of marked colimits and marked cofinality developed in this paper represents a new technology for treating such objects. Along the way to $\infty$-cofinality, we will also see that the theory of weighted colimits expounded in \cite{GHN} can be viewed as one instance of the general theory of marked colimits, and will note a fundamental relation to the Grothendieck construction, generalizing extant results for lax colimits and usual $\infty$-colimits. 

\medskip
\noindent
{\textbf{A comment from the author (22/06/20}):}
After completion of this work, a prepint \cite{Berm20} appeared proving \autoref{thm:grothendieckcomputesintro} below. This result was achieved independently by both authors using completely different methods of proof. In this work, the definition of marked colimit provided is characterized by a 2-dimensional universal property. By the time the paper \cite{Berm20} was uploaded I was working in the proof of \autoref{prop:natMO} in order to show that the definition of weighted colimits appearing in \cite{GHN} satisfies this 2-dimensional universal property. The reader not willing to take \autoref{prop:natMO} on faith can adapt the proof to the definition of lax colimit provided in the aforementioned papers.

\noindent
{\textbf{Update (28/09/20):}}
The proof of \autoref{prop:natMO} can be now found in \cite{AGSb}.

\medskip
\noindent
{\textbf{Marked colimits in 2-categories.}}
Let $\func{F: \mathbb{C} \to \mathbb{B}}$ be a 2-functor. A \emph{lax cone} for $F$ with vertex point $b \in \mathbb{B}$ is given by the following data
\begin{itemize}
	\item An object $b \in \mathbb{B}$.
	\item For every object $c \in C$ a morphism $\func{\alpha_c:F(c) \to b}$.
	\item For every morphism $\func{u:c \to c'}$ in $C$ a 2-morphism $\func{\theta_u: \alpha_c \nat  \alpha_{c'} \circ  F(u)}$ depicted by a 2-commutative diagram
	\[
	\begin{tikzcd}
	F(c)\arrow[rr, "F(u)"]\arrow[dr,bend right=5, "\alpha_c"'{name=foo}] & &|[alias=X]| F(c') \arrow[dl,bend left=5,"\alpha_{c'}"] \\
	 & b & \arrow[nat, from=foo,
	 to=X,  shorten <=.6cm,shorten >=.6cm, "\theta_u"]
	\end{tikzcd}
	\]
\end{itemize}
These data must satisfy the following set of axioms
\begin{enumerate}[label=(\Roman{*})]
	\item Unitality: $\theta_{\on{id}_c}=\on{id}_{\alpha_c}$ for every $c \in \mathbb{C}$.
	\item Composability: Given $\func{u: c \to c'}$ and $\func{v: c' \to c'' }$ the following equation holds
\[
	\begin{tikzcd}
		F(c) \arrow[dr,bend right=10,"\alpha_c"'{name=A}] \arrow[r,"F(u)"] & |[alias=B]| F(c') \arrow[d,swap,"\alpha_{c'}"'{name=C}] \arrow[r,"F(v)"] &  |[alias=D]| F(c'') \arrow[dl,bend left=10,"\alpha_{c''}"]  \\
		&  b & & \arrow[nat, from=A,
	 to=B, swap, shorten <=.2cm,shorten >=.2cm]  \arrow[nat, from=C,
	 to=D, swap, shorten <=.3cm,shorten >=.3cm]
	\end{tikzcd}
	\!  \! \! \! \! \! \! \! \! \! \! \! \! \! \! \! \! \! \! =
	\begin{tikzcd}
	F(c)\arrow[rr, "F(vu)"]\arrow[dr, bend right=5, "\alpha_c"'{name=foo}] & &|[alias=X]| F(c'') \arrow[dl,bend left=5,"\alpha_{c'}"] \\
	 & b & \arrow[nat, from=foo,
	 to=X, shorten <=.6cm,shorten >=.6cm, "\theta_{vu}"]
	\end{tikzcd}
\]
   \item Compatibility with 2-morphisms: For every 2-morphism $\func{\eta:u \nat u'}$ in $\mathbb{C}$ the following equation holds
   \[
	\begin{tikzcd}
	F(c)\arrow[rr, "F(u)"]\arrow[dr,bend right=5, "\alpha_c"'{name=foo}] & &|[alias=X]| F(c') \arrow[dl,bend left=5,"\alpha_{c'}"] \\
	 & b & \arrow[nat, from=foo,
	 to=X,  shorten <=.6cm,shorten >=.6cm, "\theta_u"]
	\end{tikzcd}
	=
	\begin{tikzcd}
	F(c)\arrow[rr,swap, bend left=30,"F(u)"'{name=A}] \arrow[rr,"F(u')"'{name=B}]\arrow[dr,bend right=5, "\alpha_c"'{name=foo}]  & &|[alias=X]| F(c') \arrow[dl,bend left=5,"\alpha_{c'}"] \\
	 & b & \arrow[nat,swap,"\theta_{u'}", from=foo,
	 to=X,  shorten <=.6cm,shorten >=.6cm]  \arrow[nat, from=A,
	 to=B,  shorten <=.2cm,shorten >=.2cm]
	\end{tikzcd}
   \]

\end{enumerate}

A morphism of lax cones $\func{\set{\alpha_{c}}_{c \in C} \to \set{\beta_c}_{c\in C}}$ with vertex point $b$ is given by a family of 2-morphisms $\set{\func{\epsilon_c: \alpha_c \nat \beta_c}}_{c \in C}$ such that for every $\func{f: c \to c'}$ the following equation holds
\[
	\begin{tikzcd}
	F(c)\arrow[rr,bend left=15, "F(f)"]\arrow[dr,bend right=40, "\alpha_c"'{name=foo}] \arrow[dr,swap,bend left=40, "\beta_c"'{name=foo2}] \arrow[nat, from=foo,
	 to=foo2,  shorten <=.3cm,shorten >=.3cm]   & & |[alias=X]| F(c') \arrow[dl,"\beta_{c'}"] \arrow[nat, from=foo2,
	 to=X,  shorten <=.3cm,shorten >=.3cm] \\
	 & b & 
	\end{tikzcd}
	=
	\begin{tikzcd}
	F(c)\arrow[rr,bend left=15, "F(f)"]\arrow[dr, "\alpha_c"'{name=foo}] & &|[alias=X]| F(c') \arrow[dl,bend right=40,"\alpha_{c'}"'{name=U}] \arrow[dl,bend left=40,swap,"\beta_{c'}"'{name=V}] \\
	 & b & \arrow[nat, from=foo,
	 to=U, swap, shorten <=.2cm,shorten >=.2cm] \arrow[nat, from=U,
	 to=V, swap, shorten <=.3cm,shorten >=.3cm]
	\end{tikzcd}
\]
One can thenn show that this defines a category of lax cones for $F$ with vertex $b$. Furthermore, we can arrange the previous definitions into a 2-functor with values in the 2-category of categories
\[
	\func{\mathbb{B} \to \mathbb{C}\!\on{at}; b \mapsto \set{\text{category lax cones with vertex }b}}
\]
The $\emph{lax colimit}$ of $F$ is then defined as an object $\colimflat\limits_{C}F \in \mathbb{B}$ correpresenting the functor above. In practice it is common to work with a category $C$ (resp. 2-category $\mathbb{C}$) equipped with a collection of chosen morphisms containing the identities  (marked categories, resp. marked 2-categories) that one wants to formally invert. It is thus desirable to have a theory of colimits adapted to accomodate the extra information present in marked (2)-categories. 

Let $\mathbb{C}^{\dagger}$ be a marked 2-category. We define a \emph{marked cone} for $F$ to be a lax cone such that the 2-morphisms $\theta_f$ are invertible whenever $f$ is marked morphism in $\mathbb{C}^{\dagger}$. Paralleling the construction above, we obtain the notion of the \emph{marked colimit} of $F$. This definition was already present in the literature under the name of $\sigma$-colimits \cite{Dubuc_limits}, where it was succesfully applied to define a 2-dimensional theory of flat pseudofunctors. Although similar in spirit, our definition of marked colimits presented in section 3 will depart from that of Descotte et al. The main cause of this difference is our use of weighted colimits as the jumping-off point of our theory, rather than using different levels of laxness (see Definition 2.4.3 in \cite{Dubuc_limits}) in the natural transformation defining a marked cone. This is because we want to our definitions to work in the setting where $\mathbb{C}$ is an $\infty$-category and $\mathbb{B}$ is an $\infty$-bicategory. In our preferred model for $\infty$-bicategories, scaled simplicial sets, we do not have direct access to different levels of laxness in the natural transformations.
\medskip
\noindent
{\textbf{Marked colimits in $\infty$-bicategories.}} Let $\goth{Cat}_{\infty}$ denote the $\infty$-bicategory of $\infty$-categories. Let $\func{F: \C \to \mathbb{B}}$ and $\func{W:\C^{\op} \to \goth{Cat}_{\infty}}$ be functors where $\C$ is an $\infty$-category and $\mathbb{B}$ is an $\infty$-bicategory. In section 3 we define $W \tensor F$, the colimit of $F$ weighted by $W$. Once this basic theory is stablished we embark upon the main construction of this work, the definition of marked colimits, appearing in section 4. Given a marked $\infty$-category $\C^{\dagger}$ we define a functor
\[
	\func{\mathfrak{C}^{\dagger}_{\C/}: \C^{\op} \to \goth{Cat}_{\infty}; c \mapsto \on{L}_{W}\left(\C^{\dagger}_{c/}\right)}
\]
where $\on{L}_{W}\left(\C^{\dagger}_{c/}\right)$ stands for the $\infty$-localization of the slice category $\C_{c/}$ with respect to the marking induced by $\C^{\dagger}$. This allows us to define the marked colimit of a functor $\func{F: \C \to \mathbb{B}}$ as $\mathfrak{C}^{\dagger}_{\C/} \tensor F$. The definition presented naturally extends the notion of $\infty$-colimits appearing in \cite{Cisinski},\cite{HTT}, as demonstrated by the following result.

\begin{theorem}\label{thm:sharpintro}
	Let $\func{F: \C \to \mathbb{B}}$ and suppose that $\colimsharp\limits_{\C}F$ exists. Then the $\infty$-categorical colimit of $F$ in the underlying $\infty$-category $\scr{B}\subseteq \mathbb{B}$ exists and there is an equivalence
	\[
		\func{\colimsharp\limits_{\C}F \to[\simeq] \colim\limits_{\C} F }
	\]
	In particular, both universal properties coincide if $\mathbb{B}=\scr{B}$.
\end{theorem}

The rest of section 4 is devoted to computational results. First, we show that weighted colimits indexed by an $\infty$-category $\C$ can be computed in terms of marked colimits. Let $\func{W:\C^{\op} \to \goth{Cat}_{\infty}}$ be a weight functor. Let us observe that its associated Cartesian fibration $\mathcal{W}$ comes equipped with a canonical marking given by the Cartesian edges. We denote this marked category by $\mathcal{W}^{\natural}$. With the aid of this observation, we can then prove:
\begin{theorem}
	Let $\func{F:\C \to \mathbb{B}}$ and $\func{W:\C^{\op} \to \goth{Cat}_{\infty}}$. Let $\func{p:\scr{W} \to \C}$ denote the Cartesian fibration classifying $W$. Suppose that $W \tensor F$ exists, then $\colimnat\limits_{\scr{W}}F \circ p$ exists and there is an equivalence in $\mathbb{B}$
	\[
		\begin{tikzcd}[ampersand replacement=\&]
			W \tensor F \arrow[r,"\simeq"] \& \colimnat\limits_{\scr{W}}F \circ p.
		\end{tikzcd}
	\]
\end{theorem}
We conclude section 4 by providing a thorough analysis of marked colimits in $\goth{Cat}_{\infty}$. To achieve this goal, we will employ the Grothendieck construction in its $\infty$-categorical  incarnation as the unstraightening functor (see chapter 3 in \cite{HTT}), as the “universal recipe” to compute colimits of diagrams in $\infty$-categories. This is witnessed by two results already known in the literature: In \cite{GHN} the authors define \emph{lax colimits} in $\infty$-categories and show in Theorem 7.4 that given a functor $F$ with values in $\infty$-categories its associated coCartesian fibration computes the lax colimit. Let us remark that our definition in the minimally marked case $\C^{\flat}$ particularizes to that of Gepner et al. If we focus our attention on the maximally marked case $\C^{\sharp}$, we note that Corollary 3.3.4 in \cite{HTT} shows that the colimit of a functor $F$ with values in $\infty$-categories can be computed as the $\infty$-localization of its associated coCartesian fibration at the collection of coCartesian edges. 

Therefore, we observe that the two extremal cases $\C^{\flat}$ resp. $\C^{\sharp}$, are already well understood. It is then natural to ask ourselves if the is a generalization of both results which can be seen through the lens of marked colimits. We provide an affirmative answer to this question in the following result.

\begin{theorem}\label{thm:grothendieckcomputesintro}
	Let $\C^{\dagger}$ be a marked $\infty$-category.  Given $\func{F: \C \to \goth{Cat}_{\infty}}$ there is a equivalence of $\infty$-categories
	\[
		 \on{L}_W \left (\on{Un}^{\on{co}}_{\C}(F)^{\natural(\dagger)}\right) \isom \colimdag_{\C}F
	\]
	where $ \on{L}_W \left (\on{Un}^{\on{co}}_{\C}(F)^{\natural(\dagger)}\right)$ denotes the $\infty$-localization at the collection of coCartesian edges lying over marked edges of $\C^{\dagger}$.
\end{theorem}

\medskip
\noindent
{\textbf{A cofinality criterion: Theorem A\textsuperscript{$\dagger$}.}} 
Let $\func{f:\C^{\dagger} \to \D^{\dagger}}$ be a marking-preserving functor. We call $f$ a \emph{marked cofinal functor} if for every diagram $\func{F:\D \to \mathbb{B}}$ the canonical comparison map\footnote{see \autoref{sec:cofinality} for a precise definition of the comparison map.}
\[
	\func{\colimdag\limits_{\C}F \circ f \to[\simeq] \colimdag\limits_{\D}F}
\]
is an equivalence in $\mathbb{B}$. The main result of this work is a characterization of this higher notion of cofinality.
\begin{theorem}\label{thm:markedcofintro}
	Let $\func{f:\C^{\dagger} \to \D^{\dagger}}$ be a marking-preserving functor of $\infty$-categories. Then $f$ is marked cofinal if and only if for every $d \in \D$ the canonical map $\func{\C_d^{\dagger} \to \D^{\dagger}_{d}}$ is induces an equivalence on localized $\infty$-categories,
	\[
		\func{\on{L}_W\left(\C_{d/}^{\dagger}\right) \to[\simeq] \on{L}_W\left(\D_{d/}^{\dagger}\right)}.
	\]
\end{theorem}
As a corollary we obtain a generalization of Theorem A in \cite{QuillenK} to marked $\infty$-categories.

\begin{corollary}[Theorem A\textsuperscript{$\dagger$}]
	Let $\func{f:\C^{\dagger} \to \D^{\dagger}}$ be a marking-preserving functor of $\infty$-categories. If the canonical map $\func{\on{L}_W\left(\C_{d/}^{\dagger}\right) \to[\simeq] \on{L}_W\left(\D_{d/}^{\dagger}\right)}$ is an equivalence of $\infty$-categories for every $d \in \D$, then the induced functor on $\infty$-localizations
	\[
		\func{f_{W}: \on{L}_W\left(\C^{\dagger}\right) \to[\simeq] \on{L}_W\left(\D^{\dagger}\right)}
	\]
	is an equivalence of $\infty$-categories.
\end{corollary}

In previous work, (see \cite{AGS} for general background and notation) we generalized Quillen's Theorem A to (strict) marked 2-categories.
\begin{theorem}\label{thm:strict2catA}
	Let $\func{F:\CC^\dagger \to \DD^\dagger}$ be a functor of marked 2-categories. Suppose that,
	\begin{enumerate}
		\item\label{asump:init} For every object $d\in \DD$, there exists a morphism $\func{g_d: d\to F(c)}$ which is initial in both $\on{L}_{W}(\CC_{d\downslash}^\dagger)$ and $\on{L}_{W}(\DD^\dagger_{d\downslash})$.
		\item\label{assump:markedinit} Every marked morphism \begin{tikzcd}
			d \arrow[r,circled] & F(c)
		\end{tikzcd}
		is initial in $\on{L}_{W}(\CC_{d\downslash}^\dagger)$. 
		\item\label{asump:presundermarked}  For any marked morphism \begin{tikzcd}f:b\arrow[r,circled] & d\end{tikzcd} in $\DD$, the induced functors $\func{f^\ast: \on{L}_{W}(\CC^\dagger_{d\downslash}) \to \on{L}_{W}(\CC^\dagger_{b\downslash})}$ preserve initial objects.
	\end{enumerate}
	Then the induced functor $\func{F_{W}: \on{L}_{W}(\Nerv_2(\CC^\dagger))\to \on{L}_{W}(\Nerv_2(\DD^\dagger))}$ is an equivalence of $\infty$-categories. 
\end{theorem} 

We also claimed that the conditions of the previous theorem should control the notion of higher cofinality for $\infty$-bicategories. This work can be understood as a partial result towards proving the \emph{cofinality conjecture} as stated in \cite{AGS}. This will be justified in  \autoref{thm:infty1dagger} where we show that the hypotheses of \autoref{thm:markedcofintro} are equivalent to the analogous conditions of \autoref{thm:strict2catA} for $\infty$-categories.

\subsection*{Acknowledgements}

I would like to thank Tobias Dyckerhoff for the guidance and support offered throughout this process. It is also a pleasure to thank Walker Stern for all the careful revisions of the draft and the many improvements suggested. The author acknowledges the support of the VolkswagenStiftung through the Lichtenberg
Professorship Programme.

\newpage

\section{Preliminaries}

In this section, we collect notation, definitions and background necessary for our constructions and proofs. We assume basic familiarity with the theory of $\infty$-categories as in \cite{Cisinski} and \cite{HTT}. We will use scaled simplicial sets as a model for $\infty$-bicategories following \cite{LurieGoodwillie}. We refer the reader to \cite{Kelly} for the basics of enriched category theory including weighted (co)limits.

\begin{notation}
	We will denote ordinary strict 1-categories by capital letters $(A,B,C)$ and $\infty$-categories by caligraphic letters $(\scr{A},\scr{B},\scr{C})$. We will generally (see \autoref{rem:cat} below for an exception) denote $\infty$-bicategories by boldface letters $(\mathbb{A},\mathbb{B},\mathbb{C})$.
\end{notation}

\begin{remark}\label{rem:cat}
	Following the previous convention we will denote by $\Cat_{\infty}$ the full subcategory of the 1-category of simplicial sets consisting of $\infty$-categories. We will use the notation $\iCat_{\infty}$ to denote the $\infty$-category of $\infty$-categories. We will denote the $\infty$-bicategory of $\infty$-categories by $\goth{Cat}_{\infty}$.
\end{remark}

\begin{notation}
	We will use extensively  marked simplicial sets as a model for $\infty$-categories. Given a simplicial set $X$ we denote by $X^{\flat} \in \msSet$ the marked simplicial set with only degenerate edges marked and by $X^{\sharp}$ the marked simplicial set with \emph{all} edges marked.
\end{notation}

\begin{notation}
	Given a simplicial set $X$ we denote by $X_{\flat} \in \on{Set}_{\Delta}^{\on{sc}}$ the scaled simplicial set with only degenerate 2-simplices being thin and by $X_{\sharp}$ the scaled simplicial set with \emph{all} 2-simplices being thin. We will identify $\infty$-categories with maximally scaled simplicial sets that are fibrant in the bicategorical model structure.
\end{notation}

\begin{notation}
	Given an $\infty$-bicategory $\mathbb{B}$ and objects $x,y \in \mathbb{B}$, we denote by $\mathbb{B}(x,y) \in \goth{Cat}_{\infty}$ the mapping category. For every $y \in \mathbb{B}$ we denote by $\mathbb{B}(\mathblank,y)$ the functor
	\[
		\func{\mathbb{B}^{\op} \to \goth{Cat}_{\infty};x \mapsto \mathbb{B}(x,y) .}
	\]
\end{notation}

\begin{notation}
	Given $\infty$-bicategories $\mathbb{B},\mathbb{C}$ and functors $\func{F,G: \mathbb{C} \to \mathbb{B}}$ we denote by $\on{Nat}_{\mathbb{B}}(F,G)$ the mapping category of $\on{Fun}\left(\mathbb{C},\mathbb{B}\right)$. We denote by $\on{Nat}_{\mathbb{B}}^{\simeq}(F,G)$ the underlying Kan complex.
\end{notation}

\begin{notation}
	Given an $\infty$-bicategory $\mathbb{B}$ we denote by $\mathcal{Y}_{\mathbb{B}}$ the $\infty$-bicategorical Yoneda embedding
	\[
		\func{\mathbb{B} \to \on{Fun}\left(\mathbb{B}^{\op},\goth{Cat}_{\infty}\right); y \mapsto \mathbb{B}(\mathblank,y).}
	\]
\end{notation}

\subsection{Marked $\infty$-categories and localizations}

\begin{definition}\label{def:minftycat}
	Let $\func{U:\msSet \to \on{Set}_{\Delta}}$ be the forgetful functor. We define $\Cat_{\infty}^{\dagger}$ (see \autoref{rem:cat}) as the full subcategory of $\msSet$ on those objects $X^{\dagger} \in \msSet$ such that $U(X^{\dagger})$ is an $\infty$-category. We call the objects of $\Cat_{\infty}^{\dagger}$ \emph{marked $\infty$-categories} and its morphisms \emph{marked functors}.
\end{definition}

\begin{definition}
	Let $\func{f: \C^{\dagger}\to \D^{\dagger}}$ be a marked functor. Given $d \in \D$ we define a marking on $\C_{d/}$ by declaring an edge $\func{\sigma: \Delta^1 \to \C_{d/}}$ to be marked if an only if its image under the canonical map is marked in $\C^{\dagger}$. We denote this marked simplicial set by $\C^{\dagger}_{d/} \in \msSet$.
\end{definition}

\begin{notation}
	Given a marked $\infty$-category $\C^{\dagger}$ we denote its $\infty$-categorical localization with respect to its marked edges by $\on{L}_{W}(\C^{\dagger})$. Given  $\EuScript{X}\in \iCat_{\infty}$ we define $\Fun^{\dagger}(\C, \EuScript{X})$ to be the full subcategory of $\Fun(\C,\X)$ on those functors mapping marked edges of $\C$ to equivalences in $\X$.
\end{notation}

\begin{remark}
	Note that the universal property of localizations implies that we have an equivalence of $\infty$-categories $\Fun^{\dagger}(\C, \EuScript{X})\isom \Fun\left(\on{L}_W (\C^{\dagger}), \EuScript{X}\right)$. We remind the reader that a model for the localization of $\C^{\dagger}$ is given by its fibrant replacement in the model structure on marked simplicial sets.
\end{remark}

\subsection{Free fibrations}
In this section we review the main results of \cite{GHN} regarding free Cartesian fibrations and their relation to marked $\infty$-categories.
\begin{definition}
	Given $\C \in \iCat_{\infty}$ let $\iCat_{\infty/ \C}^{\on{cart}}$ be the subcategory of the undercategory $\iCat_{\infty / \C}$, whose objects are Cartesian fibrations, and whose morphisms are functors which preserve Cartesian morphisms.
\end{definition}

\begin{remark}
	There is an obvious forgetful functor $\func{\mathfrak{U}: \iCat_{\infty/ \C}^{\on{cart}} \to \iCat_{\infty / \C} }$.
\end{remark}

\begin{notation}
	Given an $\infty$-category $\C$ and two Cartesian fibrations
	\[
		\func{ \X \to \C}, \enspace \func{\EuScript{Y} \to \C}
	\] 
	we denote by $\Fun^{\on{cart}}_{\C}(\X,\EuScript{Y})$ the full subcategory of functors over $\C$ that preserve Cartesian morphisms.
\end{notation}

\begin{definition}\label{def:freefib}
	Let $\C$ be an $\infty$-category. For $\func{p: \EuScript{E} \to \C}$ any functor of $\infty$-categories, let $\func{\mathfrak{F}(p): \mathfrak{F}(\E) \to \C}$ denote the map $\func{\E \times_{\C} \C^{\Delta^{1}} \to \C}$, where the pullback is along the target fibration $\func{\C^{\Delta^{1}} \to \C}$ given by evaluation at $1 \in \Delta^1$, and the projection $\mathfrak{F}(p)$ is induced by evaluation at $0$. We call the projection map $\func{\E \times_{\C} \C^{\Delta^{1}} \to \C}$ the \emph{free Cartesian fibration} on $p$. The Cartesian edges of $\mathfrak{F}(p)$ are precisely those which are mapped to equivalences under the projection to $\E$. Then $\mathfrak{F}$ defines a functor
	\[
		\func{\mathfrak{F}: \iCat_{\infty / \C}  \to \iCat_{\infty/ \C}^{\on{cart}}. }
	\]
\end{definition}

\begin{remark}
	Composition with the degeneracy map $\func{s_0: \Delta^1 \to \Delta^0}$ induces a functor $\func{\C \to \C^{\Delta^1}}$ which is a section to both of the evaluation maps. Given a functor $\func{\E \to \C}$, this section gives a natural map
	\[
		\func{ \E \to \E \times_{\C} \C^{\Delta^1}}
	\]
	over $\C$, inducing a unit natural transformation $\func{\eta:\on{id} \nat \mathfrak{U} \circ \mathfrak{F}}$.
\end{remark}

\begin{proposition}\label{prop:free2}
	Given a map of $\infty$-categories $\func{\E \to \C}$ then the unit natural transformation $\func{\E \to \mathfrak{F}(\E)}$ induces an equivalence of $\infty$-categories
	\[
		\func{\Fun^{\on{cart}}_{\C}(\mathfrak{F}(\E),\EuScript{X}) \to[\simeq] \Fun_{\C}(\E,\mathfrak{U}(\EuScript{X}))}
	\]
\end{proposition}
\begin{proof}
	This is Proposition 4.11 in \cite{GHN}.
\end{proof}

\begin{definition}\label{def:freemarked}
	Let $\func{p:\E \to \C}$ be a functor of $\infty$-categories. Suppose that $\E$ is a marked $\infty$-category and denote it by $\E^{\dagger}$. We declare an edge of $\E \times_{\C} \C^{\Delta^1}$ to be marked if and only if its projection to $\E$ is marked. We denote this marked $\infty$-category over $\C$ by $\mathfrak{F}(\E)^{\dagger}$. 
\end{definition}

\begin{remark}
	Observe that in the previous definition we can identify the fibers $\mathfrak{F}(\E)^{\dagger} \times_{\set{c}} \C$ with the marked slice $\E_{c/}^{\dagger}$ where we declare an edge marked if and only if it is marked in $\E^{\dagger}$.
\end{remark}

\begin{definition}
	Let $\func{p: \E \to \C}$ be a functor of $\infty$-categories and assume further that $\E$ is a marked $\infty$-category. Given a Cartesian fibration $\func{\X \to \C}$ we define $\Fun^{\dagger}_{\C}(\E,\X)$ to be the full subcategory on those functors mapping marked edges in $\E^{\dagger}$ to Cartesian morphisms in $\X$. If $p$ is a Cartesian fibration we define $\Fun^{\on{cart},\dagger}_{\C}(\E,\X)$ to be the full subcategory of $\Fun^{\dagger}_{\C}(\E,\X)$ on those functors which also preserve the Cartesian edges of $\E$.
\end{definition}

\begin{lemma}\label{lem:subcart}
	Let $\E^{\dagger}$ be a marked $\infty$-category together with a functor $\func{p:\E \to \C}$. Consider $\mathfrak{F}(\E)^{\dagger}$ as in \autoref{def:freemarked}. Then the unit map $\eta$ induces a commutative diagram in $\iCat_{\infty}$ 
	\[
		\begin{tikzcd}[ampersand replacement=\&]
			\Fun^{\on{cart},\dagger}_{\C}(\mathfrak{F}(\E),\X) \arrow[r,"\simeq"] \arrow[d] \& \Fun^{\dagger}_{\C}(\E,\mathfrak{U}(\X)) \arrow[d] \\
			\Fun^{\on{cart}}_{\C}(\mathfrak{F}(\E),\X) \arrow[r,"\simeq"] \& \Fun_{\C}(\E,\mathfrak{U}(\X)) 
		\end{tikzcd}
	\]
	where the vertical maps are fully faithful and the horizontal maps are equivalences of $\infty$-categories.
\end{lemma}
\begin{proof}
	The vertical maps are fully faithful by defintion. Since the bottom horizontal map is an equivalence by \autoref{prop:free2} it will suffice compute its essential image when restricted to $\Fun^{\on{cart},\dagger}_{\C}(\mathfrak{F}(\E),\X)$. It is clear that the image of this restricted map lands in $\Fun^{\dagger}_{\C}(\E,\mathfrak{U}(\X))$. Now suppose that we are given a functor of Cartesian fibrations $\func{L:\mathfrak{F}(\E) \to \X}$ such that its image under the unit map lands in $\Fun^{\dagger}_{\C}(\E,\mathfrak{U}(\X))$. Consider a marked edge in $\func{\sigma: \Delta^1 \to \mathfrak{F}(\E)}$ represented by a commutative diagram
	\[
		\begin{tikzcd}[ampersand replacement=\&]
			c \arrow[d,"u"] \arrow[r,"\beta"] \& c' \arrow[d,"v"] \\
			p(e) \arrow[r,circled,"p(\alpha)"] \& p(e')
		\end{tikzcd}
	\]
	where $\func{\alpha: \Delta^1 \to \E}$ is a marked morphism. We claim that $L(\sigma)$ is a Cartesian edge of $\X$. First we observe that we have an inner horn $\func{\theta:\Lambda^2_1 \to \X}$ given by
	\[
		\begin{tikzcd}[ampersand replacement=\&]
			c \arrow[r,"u"] \arrow[d,"u"] \& p(e) \arrow[d] \arrow[r,"p(\alpha)"] \& p(e') \arrow[d] \\
			p(e) \arrow[r,"="'] \& p(e) \arrow[r,circled,"p(\alpha)"] \& p(e').
		\end{tikzcd}
	\]
	Using Proposition 2.4.1.7 in \cite{HTT} we see that since both edges of the horn are Cartesian in $\X$ it follows that any composite of those two edges must be Cartesian in $\X$. We consider another horn $\func{\Xi: \Lambda^2_1 \to \X}$
	\[
		\begin{tikzcd}[ampersand replacement=\&]
			c \arrow[r,"\beta"] \arrow[d,"u"] \& c' \arrow[d,"v"] \arrow[r,"v"] \& p(e') \arrow[d] \\
			p(e) \arrow[r,"p(\alpha)"] \& p(e') \arrow[r,"="'] \& p(e').
		\end{tikzcd}
	\]
	First we observe that the restriction $\func{\Delta^{\{1,2 \}} \to \Lambda^1_2 \to[\Xi] \X}$ is a Cartesian edge and that the restriction of $\Xi$ to $\Delta^{\{0,1 \}}$ is $L(\sigma)$. Finally one notes that any composite of the morphisms in $\theta$ must be homotopic to any composite of the morphisms in $\Xi$. Again by Proposition 2.4.1.7 in \cite{HTT} this implies that $L(\sigma)$ is Cartesian in $\X$.
\end{proof}

\begin{definition}\label{def:fiberwiseloc}
	Let $\func{\pi:\X \to \C}$ be a Cartesian fibration and assume that $\X^{\dagger}$ is a marked $\infty$-category. A \emph{fiberwise localization} of $\pi$ at the collection of marked edges of $\X$ is a Cartesian fibration $\func{\on{L}_W^{\C}(\pi^{\dagger}) \to \C}$ together with a map of Cartesian fibrations $\func{\iota:\X \to \on{L}_W^{\C}(\pi^{\dagger})}$ such that
	\begin{itemize}
		\item The map $\iota$ sends marked edges in $\X$ to Cartesian edges in $\on{L}_W^{\C}(\pi^{\dagger})$.
		\item For any Cartesian fibration $\func{\scr{Y} \to \C}$ the induced functor $\func{\Fun^{\on{cart}}_{\C}(\on{L}_W^{\C}(\pi^{\dagger}),\Y) \to \Fun^{\on{cart},\dagger}_{\C}(\X,\Y)}$ is an equivalence of $\infty$-categories.
	\end{itemize}
\end{definition}

\begin{remark}\label{rem:cartloc}
	Given $\func{\pi:\X \to \C}$ as above let $\X^{\diamond}$ denote the marked simplicial set over $\C$ where an edge is marked if it is Cartesian \emph{or} if it is marked in $\X^\dagger$. Then a fibrant replacement in the model structure for Cartesian fibrations over $\C$ gives a model for $\on{L}_W^{\C}(\pi^{\dagger})$. 
\end{remark}

\begin{lemma}\label{lem:locbasechange}
	Let $\E^{\dagger}$ be a marked $\infty$-category together with a functor $\func{p:\E \to \C}$. Consider $\func{\mathfrak{F}(p)^{\dagger}:\mathfrak{F}(\E)^{\dagger} \to \C}$ as in \autoref{def:freemarked}. Then for every vertex $c \in \C$ there is an equivalence of $\infty$-categories
	\[
		\on{L}_{W}^{\C}\left(\mathfrak{F}(p)^{\dagger} \right)\times_{\C} \set{c} \isom \on{L}_{W}\left(\E^{\dagger}_{c/}\right).
	\]
\end{lemma}
\begin{proof}
	Let $\mathfrak{E}_{\C/}$ denote the functor classifying $\mathfrak{F}(\E)$. It is not hard to verify that this functor maps each $c\in \C$ to the $\infty$-category $\E_{c/}$ and that its action on morphisms is induced by precomposition in $\C$. Therefore we define $\widetilde{\mathfrak{E}}_{\C/}^{\dagger}$ to be the functor sending each $c$ to $\E^{\dagger}_{c/}$. It is clear that we have an equivalence of Cartesian fibrations
	\[
		\on{L}_{W}^{\C}\left(\mathfrak{F}(p)^{\dagger}\right) \simeq \on{L}_W^{\C}\left(\on{Un}_{\C}\left(\widetilde{\mathfrak{E}}_{\C/}^{\dagger}\right)\right).
	\] 
	In addition, we observe that the right-hand side can be modeled by a fibrant replacement of $\widetilde{\mathfrak{E}}_{\C/}^{\dagger}$ in the projective model structure of $\msSet$-valued functors. This finally implies 
	\[
		\on{L}_W^{\C}\left(\on{Un}_{\C}\left(\widetilde{\mathfrak{E}}^{\dagger}_{\C /}\right)\right)\times_{\C}\set{c} \simeq \on{L}_{W}\left(\E_{c/}^{\dagger}\right) \qedhere
	\]
\end{proof}

\section{Weighted colimits in $\infty$-bicategories}

\begin{definition}
	Let $\func{F: \C \to \mathbb{B}}$ and $\func{W: \C^{\op} \to \goth{Cat}_{\infty}}$. We define a functor as the composite
	\[
		\begin{tikzcd}[ampersand replacement=\&]
			\on{Nat}(W,F^{*}\mathcal{Y}_{\mathbb{B}}): \mathbb{B} \arrow[r,"\mathcal{Y}_{\mathbb{B}}"] \& \goth{Cat}_{\infty}^{\mathbb{B}^{\op}} \arrow[r,"F^{*}"] \& \goth{Cat}_{\infty}^{\C^{\op}} \arrow[rr,"\on{Nat}_{\C}(W{,}\mathblank)"] \& \& \goth{Cat}_{\infty}.
		\end{tikzcd}
	\]
\end{definition}

\begin{definition}
	Let $\func{F: \C \to \mathbb{B}}$ and $\func{W: \C^{\op} \to \goth{Cat}_{\infty}}$. We say that an object $b \in \mathbb{B}$ is the colimit of $F$ weighted by $W$ if there exists an equivalence of functors 
	\[
		\begin{tikzcd}[ampersand replacement=\&]
			\mathbb{B}(b,\mathblank) \arrow[r,Rightarrow,"\simeq"] \& \on{Nat}(W,F^{
			*}\mathcal{Y}_{\mathbb{B}}).
		\end{tikzcd}
	\]
\end{definition}

\begin{remark}
	Since weighted colimits are unique up to equivalence we will often speak of “the weighted colimit” and denote it by $W \tensor F$.
\end{remark}

\begin{definition}
	Let $\C$ be an $\infty$-category. The \emph{twisted arrow $\infty$-category} $\on{Tw}(\C)$ of $\C$ is the simplicial set given by $\Hom_{\on{Set}_{\Delta}}(\Delta^n,\on{Tw}(\C))\isom \Hom_{\on{Set}_{\Delta}}(\Delta^n \star (\Delta^n)^{\op},\C)$. Note that $\on{Tw}(\C)$ comes equipped with a projection functor $\func{\on{Tw}(\C) \to \C \times \C^{\op}}$ which is a right fibration by Proposition 5.2.1.3 in \cite{HA}, so that $\on{Tw}(\C)$ is an $\infty$-category. 
\end{definition}

We present now two of the main results obtained in \cite{AGSb} that will allow us to understand weighted colimits $\goth{Cat}_{\infty}$. The proofs can be found in the aforementioned document as Theorem 4.3 and Theorem 4.28.

\begin{proposition}\label{prop:natMO}
	Let $\C$ be a $\infty$-category and $\DD$ an $\infty$-bicategory. Then for every pair of functors $\func{F,G: \C \to \DD}$  there exists a  equivalence of $\infty$-categories
	\[
		\func{ \on{Nat}_{\C}(F,G) \to \lim_{\on{Tw}(\C)^{\op}}\on{Map}_{\DD}(F(\mathblank),G(\mathblank))}
	\]
	which is natural in each variable.
\end{proposition}

\begin{corollary}\label{cor:colimtw}
	Let $\func{F: \C \to \goth{Cat}_{\infty}}$ and $\func{W: \C^{\op} \to \goth{Cat}_{\infty}}$. Then there is an equivalence of $\infty$-categories
	\[
		\begin{tikzcd}[ampersand replacement=\&]
			W \tensor F  \arrow[r,"\simeq"] \& \colim\limits_{\on{Tw}(\C)}W \times F.
		\end{tikzcd}
	\]
\end{corollary}

We finish the section by relating conical weighted colimits, i.e. those with weight constant on the terminal category, with ordinary colimits.

\begin{proposition}\label{prop:point}
	Let $\func{F: \C \to \mathbb{B}}$ and let $\func{\underline{*}: \C^{\op} \to \goth{Cat}_{\infty}}$ denote the constant functor with value the terminal $\infty$-category. Suppose that $\underline{*}\tensor F$ exists, then $\colim\limits_{\C} F$ exists and there is an equivalence in $\scr{B}$
	\[
		\func{\underline{*}\tensor F \to[\simeq] \colim\limits_{\C} F}
	\]
	where the right-hand side is given by the $\infty$-categorical colimit of $F$ in the underlying $\infty$-category $\scr{B}\subseteq \mathbb{B}$.
\end{proposition}
\begin{proof}
	Let $\func{\Xi: \mathbb{B}(\underline{*}\tensor F,\mathblank) \nat[\simeq] \on{Nat}(\underline{*},F^{*}\mathcal{Y}_{\mathbb{B}})}$ be the natural equivalence exhibiting $\underline{*}\tensor F$ as the weighted colimit of $F$ and let $\iota^{*}\Xi$ denote its restriction to $\scr{B}$. Observe that $\iota^{*}\Xi$ factors through $\iCat_{\infty}$. We will abuse notation by viewing $\iota^{*}\Xi$ as a functor with target $\iCat_{\infty}$. Now we consider the following composite
	\[
		\func{\scr{B}\times \Delta^{1} \to[\iota^{*}\Xi] \iCat_{\infty}\to[k] \scr{S}.}
	\]
	where $k$ denotes the underlying Kan complex functor. It is clear that we have produced a natural equivalence $\func{\scr{B}(\underline{*}\tensor F,\mathblank)\nat[\simeq] \on{Nat}_{\C^{\op}}^{\simeq}(\underline{*},F^{*}\mathcal{Y}_{\scr{B}})}$ where $\mathcal{Y}_{\scr{B}}$ is the $\infty$-categorical Yoneda embedding. Observe that due to \autoref{prop:natMO} we have the following natural equivalence of spaces for every $b \in \scr{B}$
\[
	\on{Nat}_{\C^{\op}}(\underline{*},\scr{B}(F(\mathblank),b))\isom \lim\limits_{\on{Tw}(\C)^{\op}}\scr{S}(*,\scr{B}(F(\mathblank),b))\isom \lim\limits_{\on{Tw}(\C)^{\op}}\scr{B}(F(\mathblank),b)
\]
\[
	\lim\limits_{\on{Tw}(\C)^{\op}}\scr{B}(F(\mathblank),b) \isom \on{Nat}_{\C^{\op}}(F,\underline{b})
\]
where $\underline{b}$ denotes the constant functor with value $b \in \scr{B}$. We have thus produced equivalences of functors
\[
	\func{\scr{B}(\underline{*}\tensor F,\mathblank) \nat[\simeq] \on{Nat}_{\C^{\op}}(\underline{*},F^{*}\mathcal{Y}_{\scr{B}}) \nat[\simeq] \on{Nat}_{\C^{\op}}\left(F,\underline{(\mathblank)}\right).}
\]
We note that $\scr{B}_{/\underline{*}\tensor F}$ is the left fibration classifying the functor $\scr{B}(\underline{*}\tensor F,\mathblank)$ and that it has an initial object. Therefore the left fibration classifying $\on{Nat}_{\C^{\op}}\left(F,\underline{(\mathblank)}\right)$ has an initial object. This exhibits $\underline{*}\tensor F$ as the colimit of $\func{F: \C \to \scr{B}}$.
\end{proof}

\section{Marked colimits}\label{sec:markedcolimits}
\subsection{Definitions and general properties}

\begin{definition}
	Let $\C^{\dagger}$ be a marked $\infty$-category and consider $\func{\mathfrak{F}(\on{id}_{\C}):\mathfrak{F}(\C)^{\dagger} \to \C}$ as in \autoref{def:freemarked}. Given the Cartesian fibration $\on{L}^{\C}_{W}\left(\mathfrak{F}(\on{id}_\C)^{\dagger}\right)$ discussed in \autoref{def:fiberwiseloc} we let  
	\[
		\func{ \mathfrak{C}^{\dagger}_{\C /}: \C^{\op} \to \goth{Cat}_{\infty}}
	\] 
	be its associated functor.
\end{definition}

\begin{definition}\label{def:markedcolimit}
	Given a marked $\infty$-category $\C^{\dagger}$ and a diagram $\func{F: \C \to \mathbb{B}}$ we define
	\[
		\colimdag\limits_{\C}F \coloneqq \mathfrak{C}^{\dagger}_{\C /} \tensor F
	\]
	and call it the \emph{marked colimit} of $F$.
\end{definition}

\begin{remark}
	Using \autoref{cor:colimtw} we see that if $\mathbb{B}=\goth{Cat}_{\infty}$ and the marking of $\C$ consists only of equivalences our definition coincides with the definition of \emph{lax colimit} given in  \cite{GHN}. 
\end{remark}

\begin{theorem}\label{thm:sharp}
	Let $\func{F: \C \to \mathbb{B}}$ and suppose that $\colimsharp\limits_{\C}F$ exists. Then the $\infty$-categorical colimit of $F$ in the underlying $\infty$-category $\scr{B}\subseteq \mathbb{B}$ exists and there is an equivalence
	\[
		\func{\colimsharp\limits_{\C}F \to[\simeq] \colim\limits_{\C} F }
	\]
\end{theorem}
\begin{proof}
	Let $\func{t: \mathfrak{C}_{\C/}^{\sharp} \nat \underline{*}}$ denote the unique map to the terminal functor. By \autoref{prop:point} it will suffice to show that $t$ is a levelwise equivalence. This follows immediately from \autoref{lem:locbasechange} and the fact that slice categories $\C_{c/}$ are all contractible.
\end{proof}

\begin{corollary}\label{cor:sharp}
	Let $\C^{\dagger}$ be a marked $\infty$-category and consider a functor $\func{F: \C^{\dagger} \to \scr{B}}$ with $\scr{B}$ an $\infty$-category. Then $\colimdag\limits_{\C}F$ exists if and only if the colimit of $F$ exists and there is an equivalence in $\scr{B}$
	\[
		\func{\colimdag\limits_{\C}F \to[\simeq] \colim\limits_C F}
	\]
\end{corollary}
\begin{proof}
	It follows from \autoref{thm:sharp} after noting that since the mapping categories in $\B$ are Kan complexes then any natural transformation $\func{\mathfrak{C}^{\dagger}_{\C/} \nat \scr{B}(F(\mathblank),b)}$ must factor through $\mathfrak{C}^{\sharp}_{\C/}$. Thus concluding that both universal properties are the same.
\end{proof}

\begin{proposition}\label{prop:kanextensionweight}
	Let $\func{\pi:\X \to \C}$ be a Cartesian fibration where we view $\X$ as a marked category (denoted by $\X^{\natural}$) by marking the Cartesian edges. Consider the base change adjunction 	
	\[
		\pi_{!}: \iCat^{\on{cart}}_{\infty / \X} \llra : \iCat^{\on{cart}}_{\infty / \C}:\pi^{*}.
	\]
	Then for every Cartesian fibration $\func{\Y \to \C}$ there is an equivalence of $\infty$-categories
	\[
		\func{\on{Fun}_{\X}^{\on{cart}}\left(\on{L}_W^{\X}\left(\mathfrak{F}(\on{id}_\X)^{\natural}\right),\pi^{*}\Y\right) \to[\simeq] \on{Fun}_{\C}^{\on{cart}}(\X,\Y)}
	\]
	natural in $\Y$. In particular, there is an equivalence of Cartesian fibrations $\pi_! \on{L}_W^{\X}\left(\mathfrak{F}(\on{id}_\X)^{\natural}\right) \isom \X$.
\end{proposition}
\begin{proof}
	We can produce the following natural equivalences
	\[
		\on{Fun}_{\X}^{\on{cart}}\left(\on{L}_W^{\X}\left(\mathfrak{F}(\on{id}_\X)^{\natural}\right),\pi^{*}\Y\right)\isom \Fun^{\natural}_{\X}\left(\X,\mathfrak{U}(\pi^{*}\Y)\right)\isom \Fun_{\C}^{\natural}(\X,\mathfrak{U}(\Y))= \Fun_\C^{\on{cart}}(\X,\Y)
	\]
	where the third equivalence is given by the non-Cartesian base change adjunction. The result follows.
\end{proof}

\begin{theorem}
	Let $\func{F:\C \to \mathbb{B}}$ and $\func{W:\C^{\op} \to \goth{Cat}_{\infty}}$. Let $\func{p:\scr{W} \to \C}$ denote the Cartesian fibration classifying $W$. Suppose that $W \tensor F$ exists, then $\colimnat\limits_{\scr{W}}F \circ p$ exists and there is an equivalence in $\mathbb{B}$
	\[
		\begin{tikzcd}[ampersand replacement=\&]
			W \tensor F \arrow[r,"\simeq"] \& \colimnat\limits_{\scr{W}}F \circ p.
		\end{tikzcd}
	\]
\end{theorem}
\begin{proof}
	Let $\mathfrak{W}_{\scr{W}/}^{\natural}$ denote the weight functor in the definition of the colimit of $F \circ p$. By \autoref{prop:kanextensionweight} we know that $p_!\mathfrak{W}_{\scr{W}/}^{\natural} \isom W $. This shows that there is a map
	\[
		\func{\on{Nat}(W,F^{*}\mathcal{Y}_{\mathbb{B}}) \nat \on{Nat}\left(\mathfrak{W}_{\scr{W}/}^{\natural},(F \circ p)^{*}\mathcal{Y}_{\mathbb{B}}\right)}
	\]
	which is levelwise an equivalence on the underlying Kan complexes. Combining \autoref{prop:natMO} and Proposition 6.9 in \cite{GHN} we obtain for every $b \in \mathbb{B}$ the following commutative diagram
	\[
		\begin{tikzcd}[ampersand replacement=\&]
			\on{Nat}_{\C^{\op}}(W,\mathbb{B}(F(\mathblank),b)) \arrow[r] \arrow[d,"\simeq"] \&  \on{Nat}_{\scr{W}^{\op}}\left(\mathfrak{W}_{\scr{W}/}^{\natural},\mathbb{B}\left((F\circ p)(\mathblank),b\right)\right) \arrow[d,"\simeq"] \\
		\on{Fun}_{\C}^{\on{cart}}(\scr{W},\Y)	 \arrow[r,"\simeq"] \& \on{Fun}_{\scr{W}}^{\on{cart}}\left(\on{L}_W^{\scr{W}}\left(\mathfrak{F}(\on{id}_{\scr{W}})^{\natural}\right),p^{*}\Y\right)
		\end{tikzcd}
	\]
	where $\Y$ is the Cartesian fibration classifying $\mathbb{B}(F(\mathblank),b)$. Since the bottom horizontal map is an equivalence due to \autoref{prop:kanextensionweight} we conclude by 2-out-of-3.
\end{proof}

\subsection{Marked colimits in the $\infty$-bicategory of $\infty$-categories.}
In this section we show how to compute marked colimits of functors with values in $\infty$-categories. Our strategy will be a direct generalization of the arguments presented in \cite{GHN} where the authors show that the unstraightening functor computes the lax colimit of a functor.
\begin{definition}
	Let $\func{F: \C \to \iCat_{\infty}}$ be a functor and denote by $\func{\EuScript{F} \to \C}$ its associated coCartesian fibration. Given $\X \in \iCat_{\infty}$ we define a simplicial set $\Phi^{\EuScript{F}}_{\X}$ over $\C$ via the universal property $\Hom_{\C}(K, \Phi^{\EuScript{F}}_{\X})\isom \Hom (K \times_{\C}\EuScript{F},\X)$.
\end{definition}

\begin{remark}
	As an special case of (the dual of) Corollary 3.2.2.12 in \cite{HTT} we see that $\func{\Phi^{\EuScript{F}}_{\X} \to \C}$ is a Cartesian fibration. An edge $\func{\Delta^1 \to \Phi^{\EuScript{F}}_{\X}}$ is Cartesian if and only if the associated functor $\func{\Delta^1 \times_{\C} \EuScript{F}\to  \X}$ maps coCartesian edges in $\Delta^1 \times_{\C} \EuScript{F}$ to equivalences in $\X$.
\end{remark}

\begin{proposition}\label{prop:phiclassifies}
	The Cartesian fibration $\func{\Phi^{\EuScript{F}}_{\X} \to \C}$ classifies the functor
	\[
		\func{\Fun(F(-),\X):\C^{\op} \to \iCat_{\infty}.}
	\]
\end{proposition}
\begin{proof}
	This is Proposition 7.3 of \cite{GHN}.
\end{proof}

\begin{definition}\label{def:associatedmarked}
	Let $\C^{\dagger}$ be a marked $\infty$-category and consider a Cartesian (resp. coCartesian) fibration $\func{\X \to \C}$. We equip $\X$ with a marking by declaring an edge marked if and only if it is Cartesian (resp. coCartesian) and its image in $\C$ is marked. We will denote this marked $\infty$-category over $\C$ by $\X^{\natural(\dagger)}$.
\end{definition}

\begin{remark}\label{lem:markedstep1}
	Let $\C^{\dagger}$ be a marked $\infty$-category and consider a functor $\func{F: \C \to \iCat_{\infty}}$. Denote its associated coCartesian fibration by $\EuScript{F}$. Then given $\X \in \iCat_{\infty}$ we have a natural equivalence of $\infty$-categories
	\[
		\Fun\left(\on{L}_W\left(\EuScript{F}^{^{\natural(\dagger)}}\right),\X\right) \isom \Fun^{\dagger}(\EuScript{F},\X)\isom \Fun^{\dagger}_{\C}(\C,\Phi^{\EuScript{F}}_{\X}).
	\]
\end{remark}

\begin{proposition}\label{prop:cartnat}
	Let $\C$ be an $\infty$-category. Given $\func{F,G: \C^{\op} \to \Cat_{\infty}}$ classified by the Cartesian fibrations $\scr{F}$ and $\scr{G}$ respectively, there is a natural equivalence of $\infty$-categories
	\[
		\Fun^{\on{cart}}_{\C}(\scr{F},\scr{G})\isom \lim_{\on{Tw}(\C)^{\op}}\Fun(F(\mathblank),G(\mathblank))
	\]
\end{proposition}
\begin{proof}
	See Proposition 6.9 in \cite{GHN}.
\end{proof}

\begin{theorem}\label{thm:grothendieckcomputes}
	Let $\C^{\dagger}$ be a marked $\infty$-category.  Given $\func{F: \C \to \goth{Cat}_{\infty}}$ there is a equivalence of $\infty$-categories
	\[
		 \on{L}_W \left (\on{Un}^{\on{co}}_{\C}(F)^{\natural(\dagger)}\right) \isom \colimdag_{\C}F
	\]
\end{theorem}
\begin{proof}
	We fix the notation $\on{Un}^{\on{co}}_{\C}(F)^{\natural(\dagger)}=\EuScript{F}^{\natural(\dagger)}$. We have a natural equivalence of $\infty$-categories provided by \autoref{lem:markedstep1} and \autoref{lem:subcart}
	\[
		\Fun\left(\on{L}_W\left(\EuScript{F}^{\natural(\dagger)}\right),\X \right)\isom \Fun^{\dagger}_{\C}(\C,\Phi^{\EuScript{F}}_{\X}) \isom \Fun^{\on{cart},\dagger}_{\C}(\mathfrak{F}(\C),\Phi^{\EuScript{F}}_{\X})\isom \Fun^{\on{cart}}_{\C}\left(\on{L}^{\C}_W(\mathfrak{F}(\on{id}_{\C})^{\dagger}),\Phi^{\EuScript{F}}_{\X}\right)
	\]
	\autoref{prop:phiclassifies} and \autoref{prop:cartnat} in turn imply
	\[
		\Fun^{\on{cart}}_{\C}\left(\on{L}^{\C}_W(\mathfrak{F}(\on{id}_{\C})^{\dagger}),\Phi^{\EuScript{F}}_{\X}\right) \isom \lim_{\on{Tw}(\C)^{\op}} \Fun(\mathfrak{C}^{\dagger}_{\C/}, \Fun(F(\mathblank),\X))\isom \Fun\left(\colim_{\on{Tw}(\C)}\mathfrak{C}^{\dagger}_{\C/}\times F,\X \right)
	\]
	Combining these two natural equivalences, the result follows from the Yoneda lemma and \autoref{cor:colimtw}.
\end{proof}

\begin{corollary}
	Let $\C$ be an $\infty$-category. Given $\func{F: \C \to \goth{Cat}_{\infty}}$ there is an equivalence of $\infty$-categories
	\[
		 \on{Un}^{\on{co}}_{\C}(F) \isom \colimflat_{\C}F
	\]
\end{corollary}
\begin{proof}
	This follows by observing that a coCartesian edge lying over a degenerate edge must be an equivalence. 
\end{proof}

As a consequence of \autoref{thm:grothendieckcomputes}, we obtain an alternative proof of Corollary 3.3.4.3 in \cite{HTT}.

\begin{corollary}
	Let $\func{F: \C \to \iCat_{\infty}}$ then there is an equivalence of $\infty$-categories
	\[
			\on{L}_{W}\left(\on{Un}^{\on{co}}_{\C}(F)^{\natural}\right) \isom \colim_{\C}F
	\]
\end{corollary}
\begin{proof}
  Combine \autoref{thm:grothendieckcomputes} with $\C^{\sharp}$ and \autoref{prop:point}.
\end{proof}

\section{A cofinality criterion}\label{sec:cofinality}
The goal of this section is to extend the preexisting theory of cofinality to the setting of marked colimits. Our main result \autoref{thm:infty1dagger}, is a generalization of the characterization of cofinal functors appearing in Theorem 4.1.3.1 in \cite{HTT}. As an immediate corollary we obtain a generalization of Quillen's theorem A for marked $\infty$-categories. 
\begin{definition}
	Let $\func{f: \C^{\dagger} \to \D^{\dagger}}$ be a marked functor and consider $\func{\mathfrak{F}(f): \mathfrak{F}(\C)^{\dagger} \to \D}$ as in \autoref{def:freemarked}. We denote by
	\[
		\func{\mathfrak{C}^{\dagger}_{\D/}: \D^{\op} \to \goth{Cat}_{\infty}}
	\]
	the functor classifying the Cartesian fibration $\on{L}^{\D}_W\left( \mathfrak{F}(f)^{\dagger}\right)$.
\end{definition}

\begin{remark}\label{rem:Af}
	Observe that we have a natural transformation $\func{\mathcal{A}_{f}:\mathfrak{C}^{\dagger}_{\D/} \nat \mathfrak{D}^{\dagger}_{\D/}}$. We will abuse notation by also denoting by $\mathcal{A}_f$ the associated map of Cartesian fibrations.
\end{remark}

\begin{proposition}\label{prop:leftkan}
	Let $\func{f:\C^{\dagger} \to \D^{\dagger}}$ be a marked functor and let $f_{!}\dashv f^{*}$ denote the base change adjunction
	\[
		f_{!}: \iCat^{\on{cart}}_{\infty / \C} \llra : \iCat^{\on{cart}}_{\infty / \D}:f^{*}.
	\]
	Then there is a natural equivalence of $\infty$-categories
	\[
		\func{\Fun_{\C}^{\on{cart}}\left(\on{L}^{\C}_W\left( \mathfrak{F}(\on{id}_{\C})^{\dagger}\right), f^{*}\X \right) \to[\simeq] \Fun_{\D}^{\on{cart}}\left(\on{L}^{\D}_W\left( \mathfrak{F}(f)^{\dagger}\right), \X \right).}
	\]
	In particular, we have $\func{f_{!}\on{L}_W^{\C}\left(\mathfrak{F}(\on{id}_{\C})^{\dagger}\right) \to[\simeq] \on{L}^{\D}_W\left( \mathfrak{F}(f)^{\dagger}\right)}$.
\end{proposition}
\begin{proof}
	Let $\func{\X \to \D}$ be a Cartesian fibration. We observe that we have natural equivalences
	\[
		\Fun_{\C}^{\on{cart}}\left(\on{L}^{\C}_W\left( \mathfrak{F}(\on{id}_{\C})^{\dagger}\right), f^{*}\X \right) \isom \Fun_{\C}^{\dagger}(\C,\mathfrak{U}(f^{*}\X)) \isom \Fun_{\D}^{\dagger}(\C,\mathfrak{U}(\X))
	\]
	where the second equivalence is given by the non-Cartesian base change adjunction $f_* \dashv f^{*}$. It is clear that the right-hand side is equivalent to $\Fun_{\D}^{\on{cart}}\left(\on{L}^{\D}_W\left( \mathfrak{F}(f)^{\dagger}\right), \X \right)$ and the result follows.
\end{proof}

Given a marked functor $\func{f: \C^{\dagger} \to \D^{\dagger}}$ it follows that for every functor $\func{G: \D^\op \to \goth{Cat}_{\infty}}$ we can produce the following natural transformations
\[
	\begin{tikzcd}[ampersand replacement=\&]
		\on{Nat}_{\D^{\op}}(\mathfrak{D}^{\dagger}_{\D/},G) \arrow[r,nat] \& \on{Nat}_{\D^{\op}}(\mathfrak{C}^{\dagger}_{\D/},G) \arrow[r,nat,"\simeq"] \& \on{Nat}(\mathfrak{C}^{\dagger}_{\C/},f^{*}G).
	\end{tikzcd}
\]
where the last map is a natural equivalence by virtue of \autoref{prop:leftkan}. It follows that when $G=F^{*}\mathcal{Y}_{\mathbb{B}}$ for some diagram $\func{\D^{\op} \to \mathbb{B}}$, we obtain a map 
 \[
 	\func{\colimdag\limits_{\C}F\circ f \to \colimdag\limits_{\D}F }
 \]
 which we will call \emph{the canonical comparison map} whenever both are defined.
 
\begin{definition}
	A functor $\func{f: \C^{\dagger} \to \D^{\dagger}}$ of marked $\infty$-categories is said to be \emph{marked cofinal} if for every functor $\func{F:\D \to \mathbb{B}}$, the following conditions hold:
	\begin{itemize}
		\item The marked colimit $F$ exists if and only the marked colimit of $(F \circ f)$ exists.
		\item The canonical comparisonn map $\begin{tikzcd}[ampersand replacement=\&]
				\colimdag\limits_{\C}F\circ f \arrow[r,"\simeq"] \& \colimdag\limits_{\D}F
			\end{tikzcd}$ is an equivalence in $\mathbb{B}$.
	\end{itemize} 	
\end{definition}

\begin{proposition}\label{cor:square}
	Let $\func{f:\C^{\dagger} \to \D^{\dagger}}$ be a marked functor. For every diagram $\func{F:\D \to \goth{Cat}_{\infty}}$ we have a commutative diagram in $\iCat_{\infty}$ given by
	\[
		\begin{tikzcd}[ampersand replacement=\&]
			\on{L}_{W}\left( \on{Un}^{\on{co}}_{\C}(F \circ f)^{\natural(\dagger)}  \right) \arrow[r,"\simeq"]\arrow[d] \& \colimdag\limits_{\C}F \circ f \arrow[d] \\
			\on{L}_{W}\left( \on{Un}^{\on{co}}_{\D}(F)^{\natural(\dagger)}  \right) \arrow[r,"\simeq"] \& \colimdag\limits_{\D}F
		\end{tikzcd}
	\]
	where the leftmost vertical map is induced by pullback along $f$.
\end{proposition}
\begin{proof}
	Left to the reader.
\end{proof}

\begin{proposition}\label{prop:af2}
	A functor $\func{f: \C^{\dagger} \to \D^{\dagger}}$ of marked $\infty$-categories is marked cofinal if and only if $\mathcal{A}_f$ is an equivalence of Cartesian fibrations.
\end{proposition}
\begin{proof}
	Suppose $f$ is marked cofinal. Given an object $d \in \D$ we define the composite
  \[
  	\func{\mathcal{Y}_d: \D \to \scr{S}\subseteq \iCat_{\infty}}
  \]
  where first functor is given by $\on{Map}_{\D}(d,\mathblank)$ and the second functor is the inclusion of the full subcategory of spaces. Then using \autoref{cor:square} we obtain an equivalence of $\infty$-categories
	\[
		 \begin{tikzcd}
		 	\on{L}_W (\C_{d/}^{\dagger}) \arrow[r,"\simeq"] & \on{L}_{W}(\D_{d/}^{\dagger}).
		 \end{tikzcd} 
	\]
	Using \autoref{lem:locbasechange} we identify this map with the fiber of the map $\mathcal{A}_f$ over $d$. Since equivalences of Cartesian fibrations can be detected fiberwise it follows that $\mathcal{A}_f$ is an equivalence. The converse follows immediately.
\end{proof}

\begin{remark}\label{rem:sharpA}
  Suppose $\func{f:\C^{\sharp} \to \D^{\sharp}}$ is cofinal for the maximal marking.  Since $f$ is cofinal we obtain an equivalence after $\infty$-groupoid completion
  \[
  	\on{L}_{W}\left(\C_{d/}^{\sharp}\right)\isom \on{L}_W\left(\D_{d/}^{\sharp}\right).
  \]
 Then it follows that $f$ satisfies the hypothesis of Theorem 4.1.3.1 in \cite{HTT} and so pullback along $f$ preserves all $\infty$-limits.
\end{remark}

\begin{remark}
	For the rest of this section we will abuse notation by denoting the Cartesian fibration $\on{L}_{W}^{\D}\left(\mathfrak{F}(f)^{\dagger}\right)$ by $\mathfrak{C}^{\dagger}_{\D/}$ and similarly for the other fiberwise localizations of free fibrations already mentioned.
\end{remark}

\begin{lemma}
	Let $\D^{\dagger}$ be a marked $\infty$-category and consider the Cartesian fibration $\func{\mathfrak{D}^{\dagger}_{\D/} \to \D}$. Then the following hold
	\begin{itemize}
		\item For every $d \in \D$ there exists an initial object in $\mathfrak{D}^{\dagger}_{\D/} \times_{\D}\set{d}$.
		\item Every object in $\mathfrak{D}^{\dagger}_{\D/} \times_{\D}\set{d}$ represented by a marked morphism is initial.
	\end{itemize}
\end{lemma}
\begin{proof}
  The proof is analogous to the proof of Lemma 4.0.3 in \cite{AGS}.
\end{proof}

\begin{theorem}\label{thm:infty1dagger}
	Let $\func{f: \C^{\dagger} \to \D^{\dagger}}$ be a functor of marked $\infty$-categories. Then $f$ is marked cofinal if and only the following conditions hold
	\begin{enumerate}
		\item For every $d \in \D$ there is a morphism $\func{g_d:d \to f(c)}$ which is initial in $\on{L}_W\left(\C_{d/}^{\dagger}\right)$ and in $\on{L}_W\left(\D_{d/}^{\dagger}\right)$.
		\item Every object in $\on{L}_W\left(\C_{d/}^{\dagger}\right)$ represented by a marked morphism in $\D^{\dagger}$ is initial.
		\item There exists a Cartesian morphism providing a solution to the lifting problem
		\[
			\begin{tikzcd}[ampersand replacement=\&]
				\partial \Delta^1 \arrow[d] \arrow[r,"\sigma"] \& \mathfrak{C}^{\dagger}_{\D /} \arrow[d] \\
				(\Delta^1)^{\sharp} \arrow[r] \arrow[ur,dotted] \& \D^{\dagger}
			\end{tikzcd}
		\]
		if the image of $\sigma$ consists in two initial objects in their respective fibers.
	\end{enumerate}
\end{theorem}
\begin{proof}
	By \autoref{prop:af2} it will suffice to show that $\mathcal{A}_f$ is an equivalence of Cartesian fibrations precisely when the conditions above are satisfied.

	Suppose that $\mathcal{A}_f$ is an equivalence and pick an inverse $\func{\Xi:\mathfrak{D}^{\dagger}_{\D /} \to \mathfrak{C}^{\dagger}_{\D /} }$ over $\D$. First we see that the first 2 conditions can be checked by noting that $\mathcal{A}_f$ induces an equivalence of $\infty$-categories upong passage to fibers. The existence of $\Xi$ implies the existence of a section $\func{s_{f}:\D \to \mathfrak{C}^{\dagger}_{\D /}}$ mapping each object of $\D$ to an initial object in the fiber and mapping marked edges to Cartesian edges. It is routine to show that the third condition in our theorem is satisfied.
	
	To show the converse we first observe that condition (1) together with \autoref{prop:CartTrivKanFib} imply the existence of a section $\func{s_{f}:\D \to \mathfrak{C}^{\dagger}_{\D /}}$ such that for every $d \in \D$ both $s_{f}(d)$ and $\mathcal{A}_f(s_{f}(d))$ are initial in their respective fibers. We claim that $s_{f}$ maps marked edges to Cartesian edges in $\mathfrak{C}^{\dagger}_{\D /}$. To see this given $\func{\alpha:(\Delta^1)^{\sharp} \to \D^{\dagger}}$ we denote by $\widetilde{\alpha}$ the Cartesian lift provided by condition (3). One immediately checks that there exists an equivalence in $\func{\Delta^1 \to[u] \mathfrak{C}^{\dagger}_{\D/}}$ such that $s_f(\alpha)\sim u \circ \widetilde{\alpha} $ and thus $s_f(\alpha)$ is Cartesian. In particular we obtain a map of Cartesian fibrations 
\[
	\func{\Xi: \mathfrak{D}^{\dagger}_{\D / } \to \mathfrak{C}^{\dagger}_{\D /}.}
\]
	We fix the notation
	\[
		\func{\Gamma_{\D}: \Fun_{\D}^{\on{cart}}( \mathfrak{D}^{\dagger}_{\D / }, \mathfrak{D}^{\dagger}_{\D / })\to[\simeq] \Fun^{\dagger}_{\D}(\D, \mathfrak{D}^{\dagger}_{\D / })}
	\]
	and observe that both the identity functor on $ \mathfrak{D}^{\dagger}_{\D / }$ and $\mathcal{A}_f \circ \Xi$ get mapped under $\Gamma_{\D}$ to sections landing in initial objects in the fibers. It follows from \autoref{prop:CartTrivKanFib} that $\mathcal{A}_f \circ \Xi \sim id$.
	
	Similarly we consider 
	\[
		\func{\Gamma_{\C}: \Fun_{\D}^{\on{cart}}( \mathfrak{C}^{\dagger}_{\D / }, \mathfrak{C}^{\dagger}_{\D / })\to[\simeq] \Fun^{\dagger}_{\D}(\C, \mathfrak{C}^{\dagger}_{\D / })}
	\]
	and observe that  $\Gamma_{\C}(\Xi \circ \mathcal{A}_f)(c)=s_f(f(c))$ which is initial in the fiber by construction. Similarly the image of the identity functor under $\Gamma_{\C}$ sends $c \in \C$ to an object represented by a marked morphism. Condition (2) implies that this object is initial. Since both maps can factored through the pullback
	\[
		\begin{tikzcd}[ampersand replacement=\&]
			f^*\left(\mathfrak{C}_{\D/}\right) \arrow[d] \arrow[r] \&  
			 \mathfrak{C}^{\dagger}_{\D / } \arrow[d] \\
			 \C \arrow[r] \& \D
		\end{tikzcd}
	\]
	we can apply \autoref{prop:CartTrivKanFib} to their factorizations. This shows that $\Xi \circ \mathcal{A}_f \sim \on{id}$ and thus finishes the proof. 
\end{proof}

\begin{remark}
	Let $\func{\mathfrak{C}_{\D/}^{\dagger}: \D^{\op} \to \iCat_{\infty}}$. Note that condition (3) in \autoref{thm:infty1dagger} holds if and only if for every marked edge \begin{tikzcd}
		d \arrow[r,circled,"u"] & d'
	\end{tikzcd}
	in $\D$ the functor $\mathfrak{C}_{\D/}^{\dagger}(u)$ preserves initial objects.	
\end{remark}

We obtain as corollary the following generalization of Quillen's Theorem A to marked $\infty$-categories.

\begin{corollary}[Theorem $\text{A}^{\dagger}$]
	Let $\func{f:\C^{\dagger} \to \D^{\dagger}}$ be a marked cofinal functor. Then there exists an equivalence of $\infty$-categories
	\[
		\func{\on{L}_{W}\left(\C^{\dagger} \right) \to[\simeq] \on{L}_{W}\left(\D^{\dagger}\right) }
	\]	
\end{corollary}
\begin{proof}
	By \autoref{thm:grothendieckcomputes} we can identify $\on{L}_{W}\left(\D^{\dagger}\right)$ with the marked colimit of the constant point valued functor. The result then follows from \autoref{cor:square}.
\end{proof}

\begin{proposition}\label{prop:CartTrivKanFib}
	Let $\func{\pi:\X \to \D}$ be a Cartesian fibration of simplicial sets. Assume that for each vertex $c\in \C$, the $\infty$-category $\X_d$ has an initial object. Denote by $\X^\prime\subset \X$ the full simplicial subset of $\X$ spanned by those $x$ which are initial objects in $\X_{\pi(x)}$. Then 
	\[
	\func{\pi|_{\X^\prime}:\X^\prime\to \C}
	\]
	is a trivial Kan fibration of simplicial sets. Moreover, a section $s$ of $\pi:\X \to \C$ is initial in the $\infty$-category $\Fun_{\C}(\C,\X)$ if and only if $s$ factors through $\X^\prime$.  
\end{proposition}
\begin{proof}
	See Proposition 2.4.4.9 in \cite{HTT}.
\end{proof}

\end{document}